\documentclass[11pt, reqno]{amsart}
\usepackage{amssymb}
\usepackage{amsmath}
\usepackage{amsthm}
\usepackage{braket}
\usepackage{pdfpages}
\usepackage{graphicx}
\usepackage{caption}
\usepackage{subcaption}
\usepackage{mathrsfs} 
\usepackage{tikz-cd} 
\usepackage{float}
\usepackage{tabulary}
\usepackage{booktabs}
\usepackage{mathtools} 

\numberwithin{equation}{section}
\theoremstyle{plain}
\newtheorem{theorem}{Theorem}[section]

\newtheoremstyle{named}{}{}{\itshape}{}{\bfseries}{.}{.5em}{\thmnote{#3}}
\theoremstyle{named}
\newtheorem*{namedtheorem}{Theorem}

\theoremstyle{theorem}
\newtheorem{prop}[theorem]{Proposition}
\newtheorem{lem}[theorem]{Lemma}
\newtheorem{cor}[theorem]{Corollary}

\theoremstyle{definition}
\newtheorem{defn}[theorem]{Definition}
\newtheorem{rmk}[theorem]{Remark}

\newcommand{\R}{\mathbb{R}}
\newcommand{\C}{\mathbb{C}}
\newcommand{\Q}{\mathbb{Q}}
\newcommand{\Z}{\mathbb{Z}}
\newcommand{\N}{\mathbb{N}}

\newcommand{\F}{\mathscr{F}}
\newcommand{\Sym}{\mathscr{S}}
\newcommand{\me}{\mathrm{e}}

\numberwithin{figure}{section}

\begin{document}

\title{On the inhomogeneity of the Mandelbrot set}

\author{Yusheng Luo}

\address{%
Department of Mathematics, Harvard University, One Oxford Street, Cambridge, MA 02138 USA}

\email{yusheng@math.harvard.edu}

\date{29 August 2018}

\begin{abstract}
We will show the Mandelbrot set $M$ is locally conformally inhomogeneous: the only conformal map $f$ defined in an open set $U$ intersecting $\partial M$ and satisfying $f(U\cap\partial M)\subset \partial M$ is the identity map. The proof uses the study of local conformal symmetries of the Julia sets of polynomials: we will show in many cases, the dynamics can be recovered from the local conformal structure of the Julia sets.
\end{abstract}

\maketitle

\maketitle

\section{Introduction}\label{sec1}
Given a monic polynomial $f(z)$, the {\em filled Julia set} is 
$$
K=K(f)=\{z\in \C: (f^n(z))_{n\in\N} \,\mbox{ is bounded}\}
$$
and the {\em Julia set} is $J = \partial K$.
For quadratic family $f_c(z) = z^2+c$, 
the {\em Mandelbrot set} $M$ can be defined as the subset in the {\em parameter plane} such that the Julia set $J_c$ is connected,
$$
M = \{c\in\C: J_c \,\mbox{ is connected}\}
$$

\begin{figure}
  \centering
  \includegraphics[width=\textwidth]{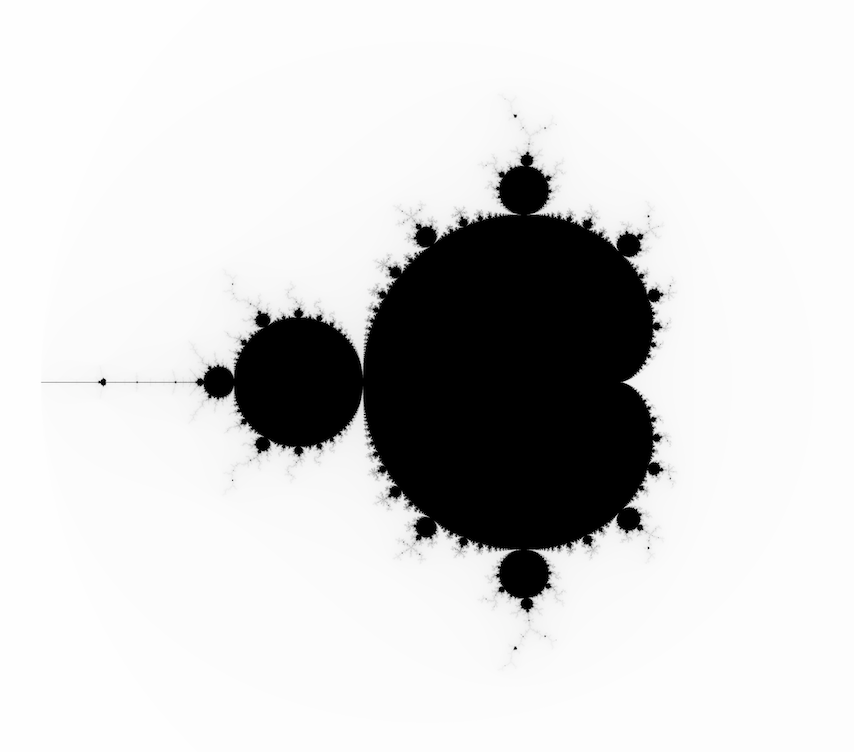}
  \caption{The Mandelbrot set}
\end{figure}

Let $A$ be a compact subset of $\C$. We call an orientation preserving homeomorphism 
$$
H:U \longrightarrow V 
$$ 
a {\em local conformal symmetry} of $A$ if $H$ is conformal, $U$ is connected and $H$ sends $U\cap A$ onto $V\cap A$.
We say a local conformal symmetry is {\em trivial} if it is the identity map or $U\cap A =\emptyset$.

In this paper, we will show
\begin{theorem}\label{maintheorem}
The boundary of the Mandelbrot set $\partial M$ admits no non-trivial local conformal symmetries.
\end{theorem}

\begin{rmk}
For a fixed $d\geq 2$, and let $f_c(z) = z^d+c$. The {\em Multibrot set} $M_d \subset \C$ is defined as the set of $c$ such that the Julia set $J(f_c)$ is connected.
The traditional Mandelbrot  set is the quadratic version $M_2$.
The above results can be easily generalized to the Multibrot sets.
The proofs are essentially the same.
\end{rmk}

As a corollary of Theorem \ref{maintheorem}, we have the following result which is proved for the Julia set of polynomials in \cite{GKN14}.
\begin{cor}\label{cor1}
The boundary of the Mandelbrot set $\partial M$ is not the Julia set of any rational map.
\end{cor}
\begin{proof}
For the Julia set $J$ of a rational map $f$, and any non-critical point $p\in J$, $f$ gives a local conformal symmetry between $p$ and $f(p)$. But $\partial M$ has no non-trivial local conformal symmetries by Theorem \ref{maintheorem}, so $J \neq \partial M$.
\end{proof}

The proof of the Theorem \ref{maintheorem} is based on the study of the symmetries of the Julia sets.
We consider the semigroup
$$
\Sigma_f := \{ g: g \text{ is a polynomial such that } g\circ f^n = f^n \circ g \text{ for some } n\in\Z_{>0}\}
$$
Inside this semigroup, we define the {\em linear symmetry group} of $f$
$$
\Sigma^*_f := \{\sigma\in \Sigma_f: \sigma \text{ is linear.}\} 
$$
and we say $f$ has {\em trivial linear symmetry group} if $\Sigma^*_f$ is trivial.
It turns out that $f$ has non-trivial linear symmetry group if and only if $f$ is conjugate to $z^sh(z^k)$ for some polynomial $h$ and positive integers $s,k$ with $(s,k) = 1$ and $k>1$ (see Section \ref{sec5}).
We also want to emphasize here that in our definition, an element in $\Sigma_f$ is only required to commute with some iterates of $f$, and may not commute with $f$.

We will show that the dynamics can be recovered in most cases from the local conformal structures of the Julia sets.
More precisely, we will prove

\begin{theorem}\label{SymJulia}
Let $f_1(z)$ and $f_2(z)$ be two monic polynomials of degree $d$ with connected and locally connected Julia set $J_1$ and $J_2$, and trivial linear symmetry groups $\Sigma^*_{f_1}= \Sigma^*_{f_2} = \{id\}$. Let
$$
S: (U_1, U_1\cap J_1) \longrightarrow (U_2, U_2\cap J_2)
$$
be a local conformal symmetry of the Julia sets with $U_1\cap J_1 \neq \emptyset$. Then either
\begin{enumerate}
\item The Julia sets are smooth, in which case $f_1(z), f_2(z)$ are conjugate to either monomial $z^d$ (which is eliminated as we assume the linear symmetry group is trivial) or $\pm$ Chebyshev polynomials $\pm T_d$,
\item The Julia sets are not smooth, in which case there exists two polynomials $P(z)$ and $Q(z)$ such 
that 
$$
S(z) = (f_2^{-n_2}\circ f_2^{m_2}) \circ P\circ (f_1^{-n_1}\circ f_1^{m_1}) (z)
$$ 
for some $n_i$ and $m_i$, and $f_1(z) = Q\circ P(z)$, $f_2(z) = P\circ Q(z)$.
\end{enumerate}
\end{theorem}

Note that the condition $\Sigma^*_{f_1}= \Sigma^*_{f_2} = \{id\}$ actually eliminates the case of Chebyshev polynomials when $d$ is odd.

In the case $f_1 = f_2$, and $x\in J$ is a periodic point of period, the set of local conformal symmetries fixing $x$ form a group under composition, which we call the {\em local conformal symmetry group at $x$}.
A similar argument also gives a classification of the symmetry group of the local symmetry at a repelling periodic point:

\begin{theorem}\label{gen}
Let $f(z)$ be a monic polynomial of degree $d$ with connected, locally connected non-smooth Julia set, and $x \in J$ be a repelling periodic point of period $p$,
then the local symmetry of the Julia set at $x$ is given by
$$
\Sym := \{f^{-pn} \circ g : n\in\N, g\in \Sigma_f, g(x) = x\}
$$
where the inverse branch is chosen so that $f^{-pn}(x) = x$.
\end{theorem}

Note that in Theorem \ref{gen}, we don't need to assume $\Sigma^*_f = \{id\}$.
The group structure allows us to show that up to post composing with the dynamics, the local symmetries all come from global polynomials.
This is not the case in general for two different polynomials if we do not have the trivial linear symmetry assumption. Consider
$$
f_1(z) = z^sR(z^{n_1})^{n_2}, f_2(z) = z^sR(z^{n_2})^{n_1}
$$ 
for some $n_1, n_2, s\in \N$ and some polynomial $R$, then $S(z) = z^{\frac{n_1}{n_2}}$ gives local symmetry between the two Julia sets, but it cannot be promoted to a global one using dynamics if $(s,n_1n_2) = 1$.

The semigroup $\Sigma_f$ or its variant has been studied intensively, and using decomposition theory of polynomials (see \cite{R22}, \cite{R23} and \cite{GTZ08}), one can classify all elements in $\Sigma_f$. 

For quadratic family $f_c(z) = z^2+c$, Theorem \ref{SymJulia} and Theorem \ref{gen} translates into the following:
\begin{namedtheorem}[Theorem \ref{SymJulia}']
Let $f_1(z) = z^2+c_1$ and $f_2(z) = z^2+c_2$ be two quadratic polynomial with connected, locally connected Julia set, which are not $z^2$ or $z^2-2$.
Let
$$
S: (U_1, U_1\cap J_1) \longrightarrow (U_2, U_2\cap J_2)
$$
be a local conformal symmetry of the Julia sets with $U_1\cap J_1 \neq \emptyset$, then $f_1=f_2$.
\end{namedtheorem}

\begin{namedtheorem}[Theorem \ref{gen}']
Let $f(z) = z^2+c$ be a quadratic polynomial with connected, locally connected Julia set, which is not $z^2$ or $z^2-2$,
then the local symmetry group at a repelling periodic point $x$ of period $p$ is generated by the $f^p$.
\end{namedtheorem}

Theorem \ref{SymJulia} allows us to prove the following theorem which will imply Theorem \ref{maintheorem} immediately:

\begin{theorem}\label{lb}
The limit models of different branch points of the Mandelbrot set are not similar.
\end{theorem}

\begin{proof}[Proof of Theorem \ref{maintheorem} assuming Theorem \ref{lb}]
To show that $\partial M$ has no local conformal symmetries, we will argue by contradiction and
assume that we have a non-trivial map $H: U \longrightarrow V$ preserving $\partial M$.
By restricting to smaller sets, we may assume that $U\cap V = \emptyset$.
Note that branch points are dense in $\partial M$, let $p$ be a branch point in $U$,
then $H(p)$ is also a branch point as $H$ is a symmetry.
Since $H$ is conformal, in particular, it is $C^1$ with $\C$-linear derivative at $p$.
This means the limit models of the branch points $p$ and $H(p)$ are similar by scaling $H'(p)$.
This gives a contradiction to Theorem \ref{lb}.
\end{proof}

\subsection{Notes and References} \label{contsym}
The study of the local conformal symmetries of the Mandelbrot set is motivated by Ghioca, Krieger and Nguyen's investigation of Dynamical Andr\'e-Oort conjecture \cite{GKN14}.
In their proof of the Dynamical Andr\'e-Oort conjecture, one of the key steps is to show that the Multibrot set is not the filled Julia set of any polynomial. 
This follows immediately from our Corollary \ref{cor1} and its generalizations to Multibrot sets.

Theorem \ref{maintheorem} is, however, no longer true if we only require the local symmetries to be homeomorphisms.

Branner and Fagella \cite{BF99} \cite{BF01} constructed a homeomorphism between the neighborhoods of two limbs (omitting the attaching point on the boundary of the main cardioid) of the Mandelbrot set with equal denominators. The homeomorphism is holomorphic on the interior and quasiconformal outside the limbs.
The construction uses quasiconformal surgery, and the homeomorphism is not directly compatible with the dynamics, so in general, the periods of hyperbolic components are changed.
It is conjectured that the homeomorphism is in fact quasi-conformal. 

There are also many homeomorphisms between different parts of the Mandelbrot sets. Douady and Hubbard \cite{DH85_p} introduced the notion of polynomial-like maps and used it to identify the homeomorphic copies ({\em small Mandelbrot sets}) of the Mandelbrot set inside the Mandelbrot set. Lyubich \cite{Lu99} showed that the homeomorphisms are quasi-conformal when the small Mandelbrot sets are {\em primitive}. The homeomorphisms are not quasi-conformal for {\em satellite} copies. Recently, Dudko and Schleicher \cite{DS12} constructed a homeomorphism between two limbs with equal denominators preserving the periods of hyperbolic components.

It is worth noting that the two constructions mentioned above do not extend to a continuous map in any neighborhood of the domain, so they don't give topologically similar points as in Branner and Fagella's construction.

\subsection{A Summary of Techniques} \label{pre}
We will need following ingredients to prove the Theorems stated in Section \ref{sec1}.
\begin{enumerate}
\item Tan Lei's Theorem (see \cite{T90}):

The Mandelbrot set $M$ is asymptotically self-similar about a Misiurewicz point $c$ with limit model given by the linearization of the Julia set of $f_c$ at the critical value $c$.

\item Branch Theorem (see \cite{Sch04} or Expos\'e XXII in \cite{DH85}):

The branch points in the Mandelbrot set are Misiurewicz points.

\item Invariant Laminations associated to polynomials (see \cite{Th87} \cite{Th09}).

\item Ritt's theory on polynomial decompositions. (see \cite{R22} \cite{E41} \cite{GTZ08}).

\end{enumerate}

We will spend the next four sections discussing these four theories. Readers who are familiar with these theories can jump to Section \ref{pf}.

\section{Tan Lei's Theorem and Asymptotic Self-Similarity}\label{sec3}

\begin{defn}[Hausdorff Distance]
Denote $\F$ the set of non-empty compact subsets of $\C$. For $A, B \in \F$, we define the semi-distance from $A$ to $B$ to be
$$
\delta(A,B) := \sup_{x\in A} d(x, B)
$$
and the Hausdorff distance of $A, B$ to be
$$
d(A,B) := \max\{\delta(A,B), \delta(B,A)\}
$$
\end{defn}

$(\F,d)$ is a complete metric space. Following \cite{T90}, we also define the {\em truncation} for a set $A \in \F$
$$
A_r:= (A\cap D_r)\cup \partial D_r
$$
where $D_r$ is the disc centered at $0$ of radius $r$.

\begin{defn}
A closed set $0\in A \in \F$ is self similar with scale $\alpha$ about $0$ if there is $r>0$ such that
$$
(\alpha A)_r = A_r
$$
A closed set $0\in B \in \F$ is asymptotically self similar with scale $\alpha$ about $0$ with a limit model $A$ if there is $r>0$ such that
$$
d((\alpha^n B)_r, A_r) \xrightarrow{n\to\infty} 0
$$

Similarly, we call a closed set $A$ is (asymptotically) self similar about $c\in A$ if the translation of the set $A-c$ is (asymptotically) self similar about $0$.
\end{defn}

If two (asymptotically) self similar sets are related via a $C^1$ diffeomorphism with $\C$-linear derivative, then their limit models are related in the obvious way:
\begin{lem}\label{limitmodelsimilar}
Let $U, V$ be two neighborhoods of $0 \in \C$, and $F:U\longrightarrow V$ is a $C^1$ diffeomorphism with $F(0) = 0$ and with derivative $T_0F = T$ $\C$-linear. Suppose that $B\subset U$ is a closed set such that $F(B)$ is asymptotically $\rho$-self-similar about $0$ with limit model $A$, then $B$ is asymptotically $\rho$-self-similar about $0$ with limit model $T^{-1}(A)$.
\end{lem}
\begin{proof}
See Proposition 2.4 in \cite{T90}.
\end{proof}

We say a quadratic polynomial $f(z)=z^2+c$ is {\em post-critically finite} if the critical point $0$ is preperiodic, and is {\em Misiurewicz} if $0$ is strictly preperiodic, i.e., it is preperiodic but not periodic. 

In \cite{T90}, Tan proved the following theorem regarding self-similarities of the Julia set and the Mandelbrot set. We will state the version for Misiurewicz points that we will be using.
\begin{namedtheorem}[Tan Lei's Theorem]\label{thm:tl}
Let $c$ be a Misiurewicz point, and $l, p$ be the minimal integers such that
$$
f_c^p(f_c^l(c)) = f_c^l(c)
$$
set $x_c = f_c^l(c)$ and $\rho = (f_c^p)'(x_c)$
\begin{enumerate}
\item The Julia set $J_c$ is asymptotically self-similar about $x_c$ with the scale $\rho$ and the limit model is $L_c=\psi(J_c\cap\bar U)$, where $\psi$ is the linearization map, and $x_c\in\bar U$ is some closed set contained in the domain of $\psi$.
\item The Mandelbrot set $M$ is asymptotically self-similar about $c$ with the scale $\rho$ and the limit model is $\frac{1}{(f_c^l)'(c)} L_c$.
\end{enumerate}
\end{namedtheorem}

It is worth noting the the same proof can be used to generalize the Tan Lei's Theorem for Multibrot sets.

\section{Branch Points and Misiurewicz Points}\label{sec2}
Let $a\in\partial M$ be a point on the boundary of the Mandelbrot set, we call it a {\em branch point} of $M$ if $M-\{a\}$ has at least $3$ connected components. Similarly, we call $a$ in a connected Julia set $J_c$ a {\em branch point} of $J_c$ if $J_c-\{a\}$ has at least $3$ connected components.
We will be using the following characterization of branch points on the Mandelbrot set. The proof can be found in \cite{Sch04}.
\begin{theorem} [Branch Theorem for Mandelbrot Set] \label{BM}
For every two post-critically finite parameters $c\neq \tilde{c}$, exactly one of the following holds:
\begin{enumerate}
\item $c$ is in the wake of $\tilde{c}$, or vice versa;
\item there is a Misiurewicz point such that $c$ and $\tilde{c}$ are in two different of its subwakes;
\item there is a hyperbolic component such that $c$ and $\tilde{c}$ are in two different of its subwakes.
\end{enumerate}
\end{theorem}
Here, with the identification of a hyperbolic component with its center, the {\em wake} of a hyperbolic component $W$ is the connected component in the complex plane separated from the origin by the two external rays landing at the root of $W$. A {\em subwake} of $W$ is the wake of a hyperbolic component other than $W$ whose root is on $\partial W$. A {\em subwake} of a Misiurewicz point $c$ is a component of the complement in $\C$ of the external rays landing at $c$ which does not contain the origin, and the {\em wake} of $c$ is the union of all subwakes together with the rays between them.

\begin{namedtheorem}[The Branch Point Theorem]\label{thm:branch}
All the branch points of the Mandelbrot set $M$ are Misiurewicz points.
\end{namedtheorem}
\begin{proof}
Let $a$ be a branch point of $M$, choose $c_1, c_2$ be two post-critically finite parameter such that $0, c_1, c_2$ are in different components of $M-\{a\}$. It can be easily shown that situation $(1)$ and $(3)$ in Theorem \ref{BM} cannot occur for $c_1, c_2$. Hence, there is a Misiurewicz point such that $c_1$ and $c_2$ are in two different subwakes. This Misiurewicz point must coincide with $a$.
\end{proof}

\begin{rmk}
Note that the above two theorems are also true for Multibrot sets, see \cite{Sch04}.
\end{rmk}

\section{Invariant Geodesic Laminations}\label{sec4}
Let $K$ be a full nondegenerate continuum in the complex plane. 
This means that $K$ is a compact connected set of cardinality greater than one and
$\C-K$ is also connected.
One can associate a Riemann mapping
$$
\phi:\C-\overline\Delta \longrightarrow \C-K
$$
normalized so that the derivative at infinity is positive.

The image of the ray $\mathcal{R}^\theta = \{r \me^{2\pi i\theta}:r>1\}$ is called the {\em (dynamical) external ray} of {\em external angle} $\theta$.
We say an external ray $\mathcal{R}^\theta$ lands $x$ if 
$$
\lim_{r\to 1^+} \phi(re^{i\theta}) = x
$$
The landing problems of external rays are related to the boundary behavior for the Riemann mapping.
If $\partial K$ is locally connected, then we have the Carath\'eodory Theorems:
\begin{namedtheorem}[Carath\'eodory's]\label{thm:c}
Let $\phi:\C-\overline\Delta \longrightarrow \C-K$ be a conformal map, then $\phi$ extends continuously to the boundary $S^1$ if and only if $\partial K$ is locally connected.
\end{namedtheorem}

Now if $f(z)$ is a monic polynomial with connected Julia set, 
the Riemann mapping $\phi:\C-\overline\Delta \longrightarrow \C-K$, usually called {\em the B\"ottcher map}, also gives a conjugation between $f$ with $z^d$, i.e., $\Psi(z^d) = f(\Psi(z))$ outside the unit disk.
If the Julia set is also locally connected, then the above conjugacy extends to a semi-conjugacy on the circle.

For monic polynomials with connected and locally connected Julia set, it is convenient to record the information of the semi-conjugacy in a geodesic lamination. Following \cite{Th87} \cite{Th09}, we will define
\begin{defn}
A {\em geodesic lamination} is a set $L$ of chords in the closed
unit disk $\Delta$, called leaves of $L$, satisfying the following conditions:
\begin{itemize}
    \item [(GL1)] elements of $L$ are disjoint, except possibly at their endpoints;
    \item [(GL2)] the union of $L$ is closed.
\end{itemize}
\end{defn}

A {\em gap} of a lamination
$L$ is the closure of a component of the complement of $\cup L$.
Any gap for a geodesic lamination is the convex hull of its intersection with the boundary
of the disk.

\begin{figure}
  \centering
  \includegraphics[width=0.45\textwidth]{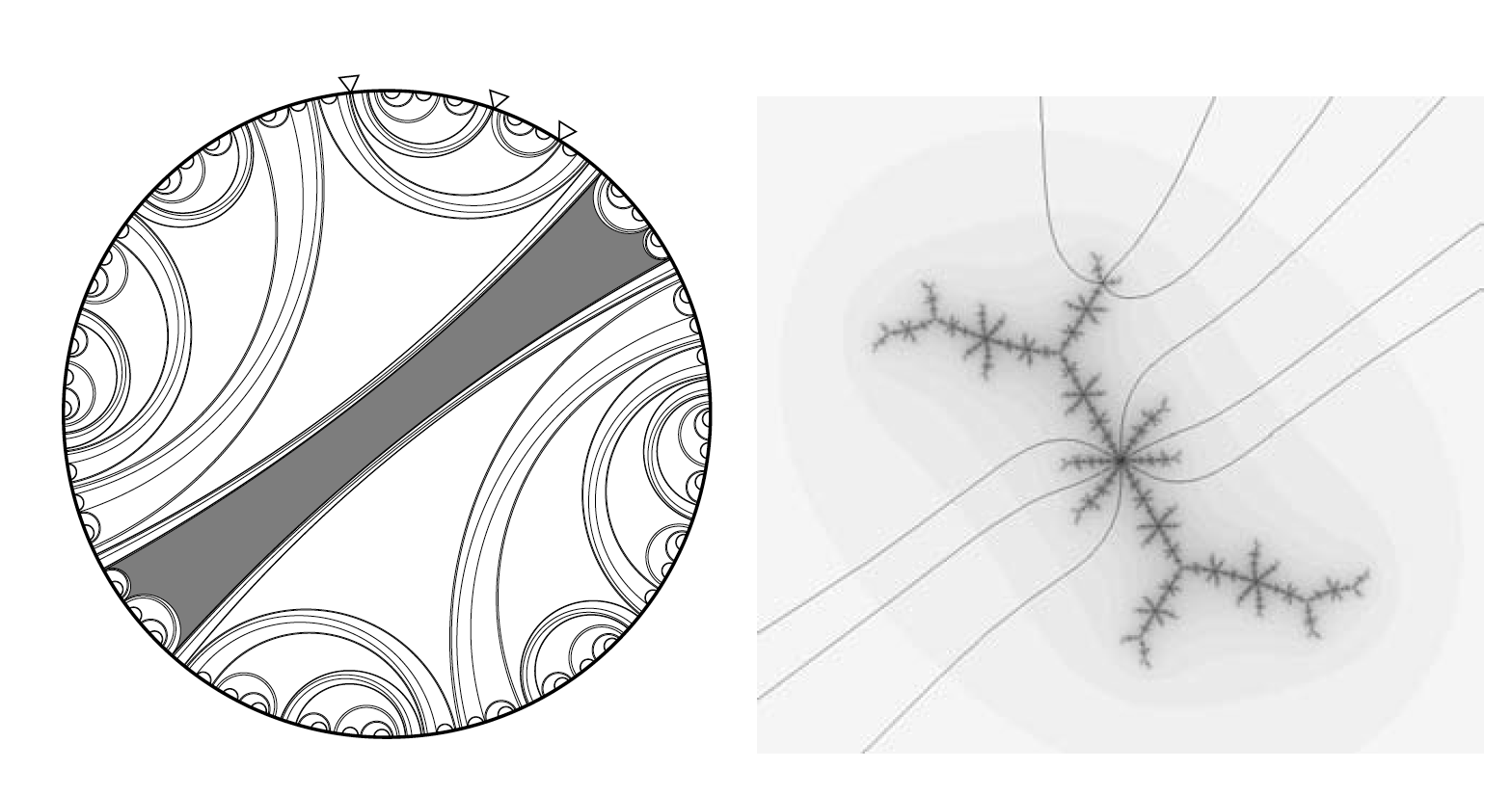} 
  \includegraphics[width=0.45\textwidth]{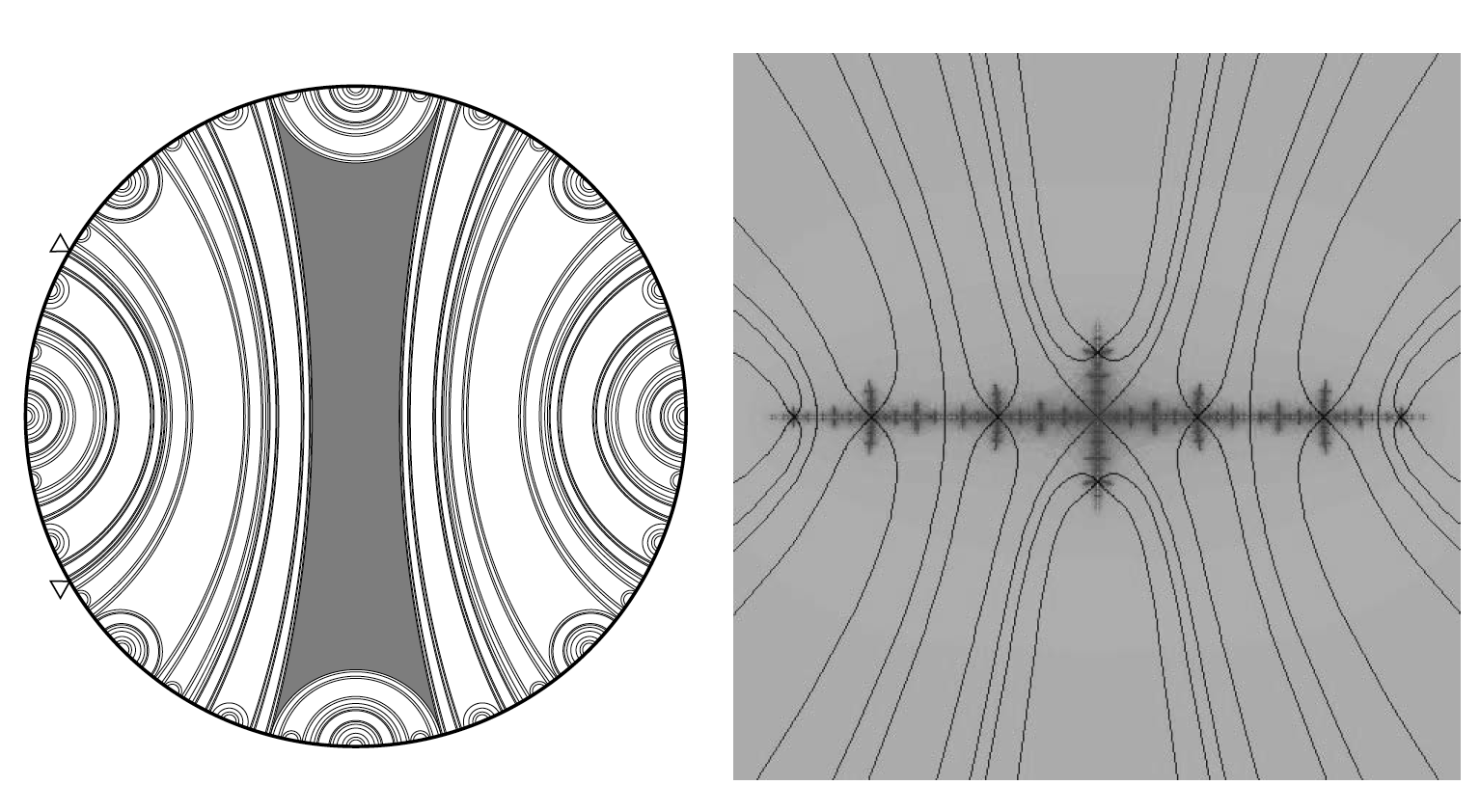}
  \caption{The quadratic lamination associated to a dendrite Julia sets, Figure taken from \cite{Th09}}
\end{figure}

\begin{defn}
Let $m_d:S^1\cong\R/\Z \longrightarrow S^1\cong\R/\Z$ be the multiplying $d$ map
$$
m_d(t) = dt
$$
A geodesic lamination is called {\em $d$-invariant} if
\begin{itemize}
    \item [(GL3)] Forward invariance: if any leaf $pq$ is in $L$, then either $m_d(p) = m_d(q)$,
or $m_d(p)m_d(q)$ is in $L$.
    \item [(GL4)] Backward invariance: if any leaf $pq$ is in $L$, then there exists a collection
of $d$ disjoint leaves, each joining a preimage of $p$ to a preimage of $q$.
    \item [(GL5)] Gap invariance: for any gap $G$, the convex hull of the image of
$G_0 = G \cap S^1$ is either a gap or a leaf or a single point.
\end{itemize}
\end{defn}

We will call a invariant geodesic lamination {\em non-trivial} if $L$ contains at least one non-degenerate leaf.

For a monic polynomial $f$ with connected and locally connected Julia set, we can associate an equivalence relation $\sim$ on the circle recording the class of points landing at the same point. $S^1/\sim$ gives a topological model for the Julia set.
Note that each equivalence class contains only finitely many points.
We can construct a geodesic lamination by forming ideal polygons for each equivalence class.
One can check that this is indeed an invariant geodesic lamination.
Note that the invariant geodesic lamination associated to $f$ with connected and locally connected Julia set is trivial if and only if the Julia set is a Jordan curve.

We will need the following fact, the proof is similar to the one for Proposition II.6.1 of \cite{Th09}:
\begin{prop}
If $L$ is an invariant geodesic lamination other than the lamination
whose leaves connect $t$ to $\frac{k}{d-1}-t$ for all $t \in \partial\overline\Delta\cong\R/\Z$, where $k=0,...,d-2$, then the gaps of $L$ are
dense in $\Delta$.
\end{prop}

Note that if $f$ is a monic polynomial of degree $d$ with connected and locally connected Julia set 
which gives the above lamination, then $f$ is post critically finite and the Julia set is topologically an interval.
Note that the conjugates of Chebyshev polynomials $T_d$ or its negative $-T_d$ give all the above lamination, so by the classification of post critical polynomials (see Theorem II of \cite{BFH92}),
$f$ is a conjugate of $T_d$ or $-T_d$.

Hence, we have
\begin{cor}\label{dense}
Let $f(z)$ be a monic polynomial of degree $d$ with connected and locally connected Julia set which is not conjugate to Chebyshev polynomials $T_d$ or $-T_d$, then the invariant geodesic lamination associated to it has dense gaps in $\Delta$.
\end{cor}

\section{Polynomial Decompositions}\label{sec5}
In our study of the local symmetries of the Julia sets, we need to consider the decomposition of a polynomial and when two polynomials commute.

In \cite{R23}, Ritt proves:
\begin{theorem}
Let $\Phi, \Psi$ be two commuting polynomials which are not conjugate to monomial or $\pm$ Chebyshev polynomial, then there exists a polynomial of the form
$$
G(z) = zR(z^r)
$$
where $R$ is a polynomial, such that $\Phi$ and $\Psi$ are simultaneously conjugate to
$\epsilon_1 G^\nu(z)$ and $\epsilon_2 G^\mu(z)$ where $\epsilon_1$ and $\epsilon_2$ are $r$-th roots of unity.
\end{theorem}

Recall that we defined the semigroup of polynomials that commute with some iterate of $f$:
$$
\Sigma_f := \{ g: g \text{ is a polynomial such that } g\circ f^n = f^n \circ g \text{ for some } n\in\Z_{>0}\}
$$
and the linear symmetry group
$$
\Sigma^*_f := \{\sigma\in \Sigma_f: \sigma \text{ is linear.}\} 
$$

By conjugation of a linear map, we may assume that $f$ is monic and centered. 
In this case, elements in $\Sigma^*_f$ are roots of unity. 
It is easy to see that the $n$-th root of unity $\zeta_n$ commute with a monic, centered polynomial $f$ if and only if $f(z) = z R(z^n)$ for some polynomial $R$.
Using polynomial decompositions, we also have the following
\begin{theorem}
Let $f$ be a monic, centered polynomial,
then $\Sigma^*_f$ is not trivial if and only if $f(z) = z^s R(z^n)$ for some polynomial $R$ and $n>1$ and $(s,n) = 1$.
\end{theorem}
\begin{proof}
First note that $f^p$ is again a monic, centered polynomial.
Since $\Sigma^*_f$ is not trivial, there exist $p$ and an $n$-th root of unity $\zeta_n \neq 1$ which commutes with $f^p$.
Hence
$$
f^p(z) = z R(z^n)
$$
By Lemma 3.11 in \cite{ZM08}, this implies that $f(z) = z^s \tilde R(z^n) \circ l$ for some polynomial $\tilde R$ and linear map $l$. Note that $(s,n) = 1$ by degree consideration.
Since $f$ is monic and centered, $l$ is the identity, so we proved one direction.

Conversely, we note that $f^p(z) = z^{s^p} P(z^n)$ for some polynomial $P$. If $(s,n) = 1$, there exists a $p$ such that $s^p \equiv1 \mod n$, so we can write $f^p(z) = z \tilde R(z^n)$ for some polynomial $\tilde R$. Hence $\Sigma^*_f$ is not trivial.
\end{proof}

\section{Proof of Theroem \ref{SymJulia}} \label{pf}
In this section, we will prove the theorem about the local symmetries of the Julia sets.
For readers' convenience, we restate the Theorem here:
{
\renewcommand{\thetheorem}{\ref{SymJulia}}
\begin{theorem}
Let $f_1(z)$ and $f_2(z)$ be two monic polynomials of degree $d$ with connected and locally connected Julia set $J_1$ and $J_2$, and trivial linear symmetry groups $\Sigma^*_{f_1}= \Sigma^*_{f_2} = \{id\}$. Let
$$
S: (U_1, U_1\cap J_1) \longrightarrow (U_2, U_2\cap J_2)
$$
be a local conformal symmetry of the Julia sets with $U_1\cap J_1 \neq \emptyset$. Then either
\begin{enumerate}
\item The Julia sets are smooth, in which case $f_1(z), f_2(z)$ are conjugate to either monomial $z^d$ (which is eliminated as we assume the linear symmetry group is trivial) or $\pm$ Chebyshev polynomials $\pm T_d$,
\item The Julia sets are not smooth, in which case there exists two polynomials $P(z)$ and $Q(z)$ such 
that 
$$
S(z) = (f_2^{-n_2}\circ f_2^{m_2}) \circ P\circ (f_1^{-n_1}\circ f_1^{m_1}) (z)
$$ 
for some $n_i$ and $m_i$, and $f_1(z) = Q\circ P(z)$, $f_2(z) = P\circ Q(z)$.
\end{enumerate}
\end{theorem}
\addtocounter{theorem}{-1}
}

\begin{rmk}\label{SJ}

\begin{enumerate}
\item Note that if the Julia set contains a smooth arc (meaning the tangent line exists at every point), then it is a smooth curve. It is known that if the Julia set is a smooth curve, then the Julia set is contained in a circle (\cite{F19} section 43), hence they are either monomial or $\pm$ Chebyshev polynomials (\cite{B91} Chapter 1).

\item In \cite{McM85} (Proposition 4.5), McMullen proves the case when $f_1(z)$ is in the main cardioid of the Mandelbrot set, (in which case $f_2(z)$ is also in the main cardioid). The strategy of our proof is similar. In the general case, we need to consider some combinatorics of Julia sets.

\item The problem of when two polynomials or rational maps have the same Julia sets has been studied extensively, see \cite{BE87} \cite{B90} \cite{L90} \cite{LP97}.
\end{enumerate}
\end{rmk}

We will now consider the case when the Julia set of $f_1(z)$ is not smooth. 
The proof consists of several lemmas.

First note that we can compose $S: (U_1, U_1\cap J_1) \longrightarrow (U_2, U_2\cap J_2)$ with the B\"ottcher maps on both sides, and get a map from open sets of the exterior of the unit disk to some other open sets. By Schwarz reflection principle, the map extends to some real analytic maps on the circle
$$
s: V_1\subset S^1 \longrightarrow V_2\subset S^1
$$
We will call $s$ the {\em associated circle map} of the local symmetry $S$.
We will also identify $S^1\cong \R/\Z$, i.e., identify a point in $S^1$ with a real number $\mod \Z$.

In general, the open sets $V_1$ and $V_2$ may be disconnected. If that's the case, we will choose a connected component and assume $V_1$ and $V_2$ are connected.

The idea is to understand how the associated circle map interacts with the original dynamics $m_d$ on the circle. Any composition of $s$ with the dynamics $m_d$ will descend to local conformal symmetries on different patches of the Julia set, and the circle is mostly used to clarify the combinatorics. 

The proof of the following lemma is based on the idea of that the symmetry of Julia set of $f_1$ and $f_2$ is discrete. The argument is classical, see \cite{L90} and \cite{LP97} for comparison.

\begin{lem}\label{RatDer}
Let $f_1(z)$ and $f_2(z)$ be two monic polynomials of degree $d$ with connected and locally connected non-smooth Julia set $J_1$ and $J_2$, and $S:U_1\longrightarrow U_2$ be a local conformal symmetry of the Julia sets.
Let $s:V_1 \longrightarrow V_2$ be the associated circle map of $S$, then
$$
s(t) = at+b.
$$
Moreover if $x\in U$ be a periodic point under $m_d(t) = dt$, then $s(x)$ is preperiodic, and hence $a,b\in \Q$.
\end{lem}
\begin{proof}

Let $m_d(t) = dt$ be the multiplication by $d$ on $S^1$, and let $x\in V_1$ be a periodic point under doubling with period $p$. 
Let $m_d^{-p}$ be the inverse branch defined on $V_1$ with $m_d^{-p}(x) = x$. We will consider the following map
$$
s_n (t) = m_d^{np} \circ s \circ (m_d^{-p})^n(t)
$$

Passing to a subsequence $s_{n_k}$, we have
$$
s_\infty(t) := \lim_{k\to\infty} s_{n_k}(t) = At+B
$$
where $A = s'(x)$.

We claim that $s_{n_k}(t) = At+B$ for all sufficiently large $k$.

First we consider the case when the Julia set is a Jordan curve.
Consider the sequence $u_{n_k} = s_{n_k}^{-1}\circ s_\infty$ which converges to the identity, and let $U_{n_k}$ be the associated local conformal symmetries.
Now composing with the dynamics if necessary, we can assume there is a repelling periodic point in the domain of $U_{n_k}$ with non-real multiplier. Such a point exists as otherwise the Julia set will be contained in a circle (\cite{ES11}), which is a contradiction to our assumption that the Julia set is not smooth.
Near the repelling periodic point, the Julia set looks like logarithmic spiral.
It is not hard to see that any map sufficiently close to identity has to fix this repelling periodic point.
Similar argument also works for all pre-images of this periodic point.
This forces the map to be identity.
Therefore, $s$ is linear, and for all large $k$, 
$$
s_{n_k}(t) = At+B.
$$

In the case when Julia set is not a Jordan curve, similar argument will also work. 
Here we provide a different and more combinatorial argument using laminations.
Since the Julia set is not a Jordan curve, the invariant lamination associated to it is not trivial.
Since we also assume the Julia set is not smooth, the gaps for the invariant lamination are dense in $\Delta$ by Corollary \ref{dense}.
If we consider a leaf with two end points in $V_1$, then it bounds some gap with all boundary points in $V_1$ by the density of the gaps.
Let $G$ be such a gap, then $s_n$ sends $G$ to some other gap.
For all large $k$, the image $s_{n_k}(G)$ is very close to $s_\infty(G)$.
Since the diameter $s_\infty(G)$ is bounded below, this forces $s_{n_k}(G) = s_\infty(G)$ for all large $k$ as the leaves are pairwise unlinked.
Therefore, the restriction of $s$ is linear on the boundary points of the gap $G$. This holds true for all gaps $G$ with boundary points in $V_1$, and $s$ is real analytic, we conclude that $s$ is linear, and for all large $k$, 
$$
s_{n_k}(t) = At+B.
$$

For the moreover part, we note that for all large $k$,
$$
s_{n_k}(x) = Ax+B,
$$
which means 
\begin{align*}
m_d^{n_kp} \circ s(x) &=  m_d^{n_kp} \circ s \circ (m_d^{-p})^{n_k}(x) \\
&= s_{n_k}(x)  = s_{n_{k+1}}(x) \\
&= m_d^{n_{k+1}p} \circ s(x) = m_d^{(n_{k+1}-n_k)p} (m_d^{n_kp} \circ s(x))
\end{align*} 
So $m_d^{n_kp} \circ s(x)$ is periodic, and 
hence, $s(x)$ is preperiodic.

The fact that $a,b\in \Q$ now follows immediately from the fact $s$ sends periodic points to the preperiodic points.
\end{proof}

\begin{lem}\label{ext}
Let $f_1(z)$ and $f_2(z)$ be two monic polynomials of degree $d$ with connected and locally connected non-smooth Julia set $J_1$ and $J_2$, and $S:U_1\longrightarrow U_2$ be a local conformal symmetry of the Julia sets.
Let 
\begin{align*}
s:V_1 &\longrightarrow V_2\\
t&\mapsto at+b.
\end{align*}
be the associated circle map of the local symmetry $S$. Let $x \in U$ be a periodic point of $m_d$ of period $p$.

If $a$ is a positive integer, and $s(x)$ is also a periodic point of $m_d$ of period $p$, then the local symmetry $S$ extends to a polynomial of degree $a$ and $S\circ f_1^p = f_2^p \circ S$.
\end{lem}
The strategy is that if $a$ is an integer, the symmetry $S$ has an analytic continuation to $\C-K_1$, where $K_1$ is the filled Julia set of $f_1$. If $s(x)$ is also periodic of period $p$, then $S$ is a local conjugacy between $f_1^p$ and $f_2^p$. Now composing with dynamics and the expanding property of Julia set, one can extend $S$ to an analytic map on $\C$ and get a conjugacy between $f_1^p$ and $f_2^p$. The map is a polynomial as it only has a pole at infinity.

\begin{proof}
Let $z_1 \in J_1$ be the landing point of the ray $\mathcal{R}^x$.
Note that $z_1$ and $S(z_1)$ are both periodic, with periods dividing $p$. They are either repelling or parabolic periodic points.

Note that if there are more than $1$ external rays landing at $z_1$, a neighborhood of $z_1$ contains external angles outside of $U$ (which we choose to be a connected component). 
A priori, in different component, the associated circle map may have different derivative or translation distance. 
To deal with this situation, let $L$ be a leaf of the invariant lamination of $f_1$ in $V_1$, and $\tilde V_1 \subset S^1$ be the side of the $L$ which entirely contained in $V_1$. 
Let $\tilde{x}$ be a point in $\tilde V_1$ which eventually maps to the periodic orbit of $x$, and denote its landing point as $\tilde{z}_1$.
Then all the landing angles of $\tilde{z}_1$ are contained in $\tilde V_1$ as the leaves are pairwise unlinked.
Now using the dynamics $f_1$ to move $\tilde x$ to $x$, and the dynamics $f_2$ to control the derivative, we can assume that the associated circle map is of the form $s(t) = at+b$ for all connected neighborhood in $S^1$ of landing angles of $z_1$.

Since the associated circle map is $s(t) = at+b$, we know outside the filled Julia set $K_1$, $S|_{U_1-K_1}$ is the restriction of the map
$$
\varphi:= \Psi_2 \circ R \circ \Psi_1^{-1}
$$
where $\Psi_j :\C-\Delta \longrightarrow \C-K_j$ is the B\"ottcher map of $f_j$, and 
$$
R(z) = e^{2\pi bi} z^a.
$$

Let $m_d^{-p}$ be the inverse branch defined on $V_1$ with $m_d^{-p}(x) = x$, and since $s(x)$ has period $p$ as well, we have
$$
s(x) = m_d^{p} \circ s \circ (m_d^{-p})(x)
$$
Hence, $s(t) = m_d^{p} \circ s \circ (m_d^{-p})(t)$ for all $t\in V_1$ as the maps on both sides are linear with the same derivative and agree at a point. Therefore, the local symmetry $S$ is a conjugation between $f_1^p$ and $f_2^p$.

Hence we have the commuting diagram
\[
\begin{tikzcd}
  U_1\cup(\C-K_1) \arrow[r, "S"] \arrow[d, "f_1^p"]
    & U_2\cup (\C-K_2) \arrow[d, "f_2^p"] \\
  f_1^p(U_1)\cup(\C-K_1) \arrow[r, "S"]
&f_2^p(U_2)\cup(\C-K_2)
\end{tikzcd}
\]
We will now use dynamics to extend $S$ to $\C$.

Shrink $U_1$ if necessary, we may assume $U_1$ is a Jordan domain, $\overline{U_1} \subset f_1^p(U_1)\cup \{z_1\}$ and $f_1^p$ is an isomorphism of $U_1$ and $f_1^{p} (U_1)$ and
$$
\cap_{k\geq 0} f^{-k} (\overline{U_1}) = \{z_1\}.
$$
Since $U_1$ intersects the Julia set non-trivially, and the filled Julia set is compact, there is an $N$ such that
$$
K_1 \subset \cup_{i=0}^N f_1^{np}(U_1) = f_1^{Np}(U_1)
$$
Denote $F(z) := f_1^{Np}(z)$ and $G(z) := f_2^{Np}(z)$. 
Let $Q$ be the critical values of $F$, and $\Omega$ be the union of all Fatou sets.
Perturb $U_1$ if necessary, we may also assume that all the intersection points of $F(\partial U_1)$ consists of two transversal arcs, and $F(\partial U_1) \cap Q = \emptyset$.
Given a point $z\in F(U_1)$, we define
$$
S(z) = G \circ S(y)
$$
where $y\in U_1$ and $F(y) = z$.
We will show that $S$ is well-defined and is holomorphic.

First, note that since $S$ conjugates $f_1^p$ and $f_2^p$ on $\C-K_1$, so $S$ is well defined outside of the filled Julia set and extends to $\varphi$.

Secondly, we will show that $S$ is well defined in a neighborhood $N(J_1)$ of the Julia set.
Consider the set
\begin{align*}
W:= \{z: &\exists\text{ a simple arc } \gamma \text{ connecting } z \text{ to } J_1 \text{ with } \mathring\gamma \subset \Omega-Q \\
&\text{ such that } \forall y\in \overline{U_1} \text{ with } F(y) = z, \\
&\text{ the component of } F^{-1}(\gamma) \text{ containing } y \text{ is contained in }U_1 \\
&\text{ (except possibly at y).} \}
\end{align*}
Here $\mathring\gamma$ denotes the interior of the arc $\gamma$.

We first note that $S$ is well-defined on $W-Q$. Let $z\in W$, and fix $\gamma$ to be a simple arc as in the definition of $W$, then for any $y\in U_1$ with $F(y) = z$, let $\gamma'$ be the component of $F^{-1}(\gamma)$ through $y$.
By the definition of $\gamma$, we have $\gamma' \subset U_1$.
We may choose a tubular neighborhood $N$ of $\gamma$ and corresponding neighborhood $N'$ of $\gamma'$ and assume that $F$ gives isomorphisms between $N'$ and $N$.
Using the isomorphism, we get an analytic function $G\circ S\circ F^{-1}$ on $N$.
Note that $N$ intersects $\C-K_1$, and $G\circ S\circ F^{-1} = \varphi$ on $N-K_1$ hence
$$
G\circ S(y) = \varphi|_{N}(z)
$$
where $\varphi|_{N}(z)$ is the value at $z$ of the unique analytic extension of $\varphi$ on $N$.
Since for different preimages of $z$ in $U_1$, by definition of $W$, the same arc $\gamma$ is used to construct the analytic map, so $S$ is well-defined.

It is not hard to check that $W$ is open and contains the Julia set, hence $S$ is well defined $N(J_1)- Q$, where $N(J_1)$ is a neighborhood of the Julia set $J_1$. But note that $Q$ is finite and nearby points have bounded image, so they are removable singularities. This proves the claim.

Lastly, we will show $S$ is well-defined in bounded Fatou components. 
Let $\Omega_0$ be a Fatou component and
let $z \in \Omega_0 - Q - A$ where $A$ is the set of non-repelling periodic points.
Fix $\gamma$ to be a simple arc connecting $z$ to the Julia set, with $\mathring\gamma \subset \Omega_0$ contains no points in $A$ or post-critical points of $F$, then for any $y_0\in U_1$ with $F(y_0) = z$,
let $\gamma_0$ be the component of $F^{-1}(\gamma)$ containing $y_0$, and $\gamma_k$ be the unique component of $F^{-1}(\gamma_{k-1})$ containing $y_k \in U_1$ and $F(y_k) = y_{k-1}$ with the property that
$$
f_1^p(y_k),...,f_1^{Np} (y_k) = F(y_k) \subset U_1.
$$
Note that this implies that $G\circ S(y_k) = S(y_{k-1})$ and $y_k \to z_1$.

We claim that $\gamma_k$ eventually lies in $U_1 \cup N(J_1)$.

To see this, note that if $\gamma_k$ lies in some strictly preperiodic Fatou component for some $k$, then $\gamma_k$ eventually lies in $U_1\cup N(J_1)$ as $K_1-N(J_1)$ is compact, so it is covered by only finitely many Fatou components.

Hence, we may assume, by replacing $F$ with some iterates of $F$, that $\gamma_k$ lies a fixed Fatou component.
This Fatou component cannot be a Siegel disk, as we assume the Julia set is locally connected,
so there are no periodic points on the boundary.
Hence we may assume that the Fatou component is either parabolic or (super-)attracting.
Let $\epsilon = d(\gamma, A) > 0$, then there exists an $n$ such that
$$
F^n(\Omega_0-N(J_1)) \subset \cup_{p\in A} B(p,\epsilon)
$$
as the $F^k$ converges locally uniformly to $A$ in $\Omega_0$.
Hence, $\gamma_n \subset N(J_1)$.

Since $\mathring\gamma$ contains no postcritical points, we may choose a tubular neighborhood $N$ of $\gamma$ and corresponding neighborhood $N'$ of $\gamma_n$ and assume that $F^n$ gives isomorphisms between $N'$ and $N$.
Using the isomorphism, we get an analytic function $G^n\circ S\circ F^{-n}$ on $N$.
Note that $N$ intersects $\C-K_1$, and $G^n\circ S\circ F^{-n} = \varphi$ on $N-K_1$ hence
$$
G^n\circ S(y_n) = \varphi|_{N}(z)
$$
where $\varphi|_{N}$ is understood as the unique analytic extension of $\phi$ on $N$.
Note that by construction, $G\circ S(y_k) = S(y_{k-1})$, hence we have
$$
G\circ S(y_0) = \varphi|_{N}(z).
$$
Since for different preimages of $z$ in $U_1$, by definition of $W$, the same arc $\gamma$ is used to construct the analytic map, so $S$ is well-defined.

Note that since $Q\cup A$ is discrete and nearby points have bounded image, so they are removable singularities. This proves the claim.

Hence we get an entire function $S(z)$ conjugating $f_1^p$ and $f_2^p$. The map is a polynomial of degree $a$  as it has a single pole at infinity of order $a$.

This proves the lemma.
\end{proof}

Note that a point $\frac{p}{q}\in S^1\cong \R/\Z$ with $(p,q) = 1$ is periodic under $m_d$ if and only if $(q,d) = 1$.
Composing with dynamics, we can assume that the derivative of the associated circle map is of the form $\frac{u}{v}\cdot l$ where $(v,d) = (u,d) = 1$ and $l\mid d$.
With this normalization of the derivative, we note that $s$ sends periodic points of $m_d$ to periodic points of $m_d$ if and only if it sends one periodic point to a periodic point. 
Hence, composing with dynamics, we can also assume that $s$ sends periodic points to periodic points.

We will call $s(t) = \frac{ul}{v}  \cdot t+b$ with the above assumption the {\em normal form} of the associated circle map. 
We will also call a local conformal symmetry $S$ is in the {\em normal form} if the associated circle map is.

Note that if $l=1$, then $s^{-1}(t)$ is also in the normal form, otherwise, $m_d\circ s^{-1}$ is the normal form.

\begin{lem}\label{lem:int}
Let $f_1(z)$ and $f_2(z)$ be two monic polynomials of degree $d$ with connected and locally connected Julia set $J_1$ and $J_2$, and trivial linear symmetry groups $\Sigma^*_{f_1}= \Sigma^*_{f_2} = \{id\}$. Let $S:U_1\longrightarrow U_2$ is a local conformal symmetry of the Julia set in the normal form with the associated circle map $s:V_1\subset S^1 \longrightarrow V_2\subset S^1$ of the form 
$$
s(t) = \frac{ul}{v}\cdot t+b.
$$
Let $x\in V_1$ be a periodic point with period $p$, then its image $s(x) \in V_2$ also has period $p$.
Moreover, $v = u = 1$.
\end{lem}
\begin{proof}
We will first show that the image also has period $p$. Assume that the period of $s(x)$ is $q$, 
and consider the map
$$
\phi = [m_d^q, s] = m_d^{-q} \circ s^{-1}\circ m_d^{q}\circ s
$$
where $m_d^{-q}$ is chosen so that $m_d^{-q}(x)$ is in the periodic orbit of $x$.
Note that the derivative is $1$ and $x, \phi(x)$ are in the same orbit, so by Lemma \ref{ext}, $\phi$ is the associated circle map of a local conformal symmetry which extends to a linear map $\Phi(z) = Az+B$ and $\Phi\circ f_1^p(z)  =  f_1^p(z)\circ \Phi$.
Since $f_1(z)$ has trivial symmetry group, this means $\Phi(z) = z$. Therefore, $\phi(x) = x$, so $p$ divides $q$.
The same argument applies to the inverse of $s$, and we get $p=q$.

For the moreover part,
note that if $v\neq 1$, we can find a periodic point $y \neq x$ whose image has different period, which gives a contradiction.
By considering the inverse map, and put it in the normal form, and apply the same argument, we conclude that $u=1$, proving the lemma.

\end{proof}

\begin{proof}[Proof of Theorem \ref{SymJulia}]
If the Julia set is smooth, by Remark \ref{SJ}, they are conjugates of $z^d$ or $\pm T_d$ where $T_d$ is the Chebyshev polynomials.

Hence, we only need to consider when the Julia set is not smooth.

Let $s:V_1\longrightarrow V_2$ be the associated circle map of the local conformal symmetry, and let $\tilde x\in U$ be in the grand orbit of the fixed point $x$ (under $m_d$).
Composing with dynamics, we can assume that $s(t)$ is in the normal form and $x\in V_1$.
By Lemma \ref{lem:int}, $s(t) = l\cdot t+b$ where $l\mid d$ and $s(x)$ is also a fixed point. Hence, by Lemma \ref{ext}, we know $S$ extends to a polynomial of degree $l$ and $P\circ f_1 (z)= f_2 \circ P(z)$.

If $l = 1$, we conclude that $P$ gives a conjugacy between $f_1$ and $f_2$, in which case, we may take 
$$
Q(z) = f_1 \circ P^{-1}(z)
$$
and conclude the result.

If $l\neq 1$, then apply the same argument to the normal form of the inverse $m_d\circ s^{-1}$, we get a polynomial $Q$ such that 
$$
Q\circ f_2(z) = f_1 \circ Q(z)
$$
Note that $(m_d\circ s^{-1})\circ s = m_d$, so
$$
Q\circ P(z) = f_1(z)
$$

Hence, we have $f_1(z) = Q\circ P(z)$ and $f_2(z) = P\circ Q(z)$.

\end{proof}

\section{Proof of Theorem \ref{gen} and \ref{lb}} \label{pfgen}

For readers' convenience, we restate the Theorem \ref{gen} here:
{
\renewcommand{\thetheorem}{\ref{gen}}
\begin{theorem}
Let $f(z)$ be a monic polynomial of degree $d$ with connected, locally connected non-smooth Julia set, and $x \in J$ be a repelling periodic point of period $p$,
then the local symmetry of the Julia set at $x$ is given by
$$
\Sym := \{f^{-pn} \circ g : n\in\N, g\in \Sigma_f, g(x) = x\}
$$
where the inverse branch is chosen so that $f^{-pn}(x) = x$.
\end{theorem}
\addtocounter{theorem}{-1}
}

\begin{proof}
Let $S$ be a local symmetry of the Julia set at a repelling periodic point $x$.
Let $s(t) = \frac{u}{v}t+b$ be the associated circle map of $S$.
Post compose with the dynamics $f^p$, we may assume that $(v,d) = 1$.
If $v\neq 1$, then $\frac{u}{v}, d^p$ generate a dense subgroup of $\R_{>0}$. Hence, by taking limit, $s$ and $m_d^p$ generate a linear map with irrational derivative. This is not possible by Lemma \ref{RatDer}. Hence, $v=1$. Then by Lemma \ref{ext}, we know $S$ extends to a polynomial map and $S\circ f^p = f^p \circ S$ with $S(x) = x$, proving the Theorem.
\end{proof}

By conjugation, we may assume that $f$ is monic and centered.
By Ritt's Theorem on commuting polynomials (see section \ref{sec5}), if $g\in \Sigma_f$,
there exists a monic polynomial of the form $h(z) = z R(z^k)$, with $k \geq 1$, a $k$-th root of unity $\zeta$ and two positive integers $a, b$ such that
$$
g(z) = \zeta h^a(z), f^p(z) = h^b(z)
$$

If $f$ is not an iterate of a polynomial and $\Sigma_f^*$ is trivial, then $h(z) = f^l(z)$ for some $l$, and $k=1$. Hence $g$ is an iterate of $f$.

Theorem \ref{gen}' now follows immediately from Theorem \ref{gen} as the linear symmetry groups of a quadratic polynomial $f(z) = z^2+c$ is trivial when $c\neq 0$, and $f$ is not an iterate of any polynomial.

We will now prove Theorem \ref{lb}. For readers' convenience, we restate the Theorem \ref{lb} here.
{
\renewcommand{\thetheorem}{\ref{lb}}
\begin{theorem}
The limit models of different branch points of the Mandelbrot set are not similar.
\end{theorem}
\addtocounter{theorem}{-1}
}
\begin{proof}
Let $c_1$ and $c_2$ be two branch points of the Mandelbrot set with the same limit models.
By the Branch Theorem, $c_1$ and $c_2$ are Misiurewicz points. By Tan Lei's Theorem, the Mandelbrot is asymptotically self-similar at $c_1$ and $c_2$ with limit models given by the limit models of the critical periodic orbits of $f_{c_1}$ and $f_{c_2}$.
Since both $f_{c_1}$ and $f_{c_2}$ are post critically finite, their Julia sets are connected and locally connected.
By composing with the two linearization map, the similarity between the two limit models gives a local conformal symmetry between the two Julia sets $J_{c_1}$ and $J_{c_2}$.
Since the linear symmetry groups of $f(z) = z^2+c$ is trivial when $c\neq 0$,
by Theorem \ref{SymJulia}, the two polynomials $f_{c_1}$ and $f_{c_2}$ are conjugate to each other, hence $c_1 = c_2$.
\end{proof}

\section{Acknowledgments}
The author thanks C. T. McMullen for advice and helpful discussion on this problem.


%
%
%
%
%
%
%
\end{document}